\documentclass[psamsfonts, draft]{amsart}
\usepackage{amssymb,amsmath}
\usepackage{amsthm}
\usepackage{mathrsfs}
\usepackage{enumerate, indentfirst}
\usepackage[fleqn,tbtags]{mathtools}
\usepackage{courier}
\usepackage{hyperref}

\usepackage{xr}

\hypersetup{
    citecolor=black,%
    filecolor=black,%
    linkcolor=black,%
    pdffitwindow=true,%
    colorlinks=false,%
    unicode=true,
    }
\newtheoremstyle{mytheorem}%
{5pt}%
{3pt}%
{\itshape}%
{1pt}%
{\bf}%
{.}%
{.5em}%
{}%

\newtheoremstyle{myremark}%
{5pt}%
{3pt}%
{\upshape}%
{1pt}%
{\em}%
{.}%
{.5em}%
{}%

\newtheoremstyle{myexample}%
{5pt}%
{3pt}%
{\upshape}%
{1pt}%
{\bf}%
{.}%
{.5em}%
{}%

\theoremstyle{mytheorem}
\newtheorem{theorem}{Theorem}[section]

\newtheorem{lemma}[theorem]{Lemma}

\newtheorem{corollary}[theorem]{Corollary}

\theoremstyle{myremark}

\theoremstyle{myexample}

\numberwithin{equation}{section}
\makeatletter \renewenvironment{proof}[1][\proofname] {\par\pushQED{\qed}\normalfont\topsep6\p@\@plus6\p@\relax\trivlist\item[\hskip\labelsep\itshape #1\@addpunct{.}]\ignorespaces}{\popQED\endtrivlist\@endpefalse} \makeatother

\renewcommand{\phi}{\varphi}
\renewcommand{\theta}{\vartheta}

\DeclareMathOperator{\mul}{mul}

\newcommand{\alg}{\mathscr{A}}

\newcommand{\abs}[1]{\lvert#1\rvert}
\newcommand{\dupN}{\mathbb{N}}

\newcommand{\seq}[1]{(#1_{n})_{n\in\dupN}}

\newcommand{\dupC}{\mathbb{C}}
\newcommand{\pia}{\pi_A}
\newcommand{\pib}{\pi_B}
\newcommand{\piba}{\pi_{B,a}}
\newcommand{\pibs}{\pi_{B,s}}

\newcommand{\ran}{\operatorname{ran}}

\newcommand{\bha}{\mathscr{B}(\hila)}

\newcommand{\M}{\mathfrak{M}}

\newcommand{\hil}{H}
\newcommand{\hila}{H_A}
\newcommand{\hilb}{H_B}

\newcommand{\kil}{K}
\newcommand{\bh}{\mathscr{B}(\hil)}

\DeclarePairedDelimiterX\sip[2]{(}{)}{#1\,\delimsize\vert\,#2}
\DeclarePairedDelimiterX\siptilde[2]{(}{)_{\!_{\widetilde{A}}}}{#1\,\delimsize\vert\,#2}
\DeclarePairedDelimiterX\sipn[2]{(}{)_{\nu}}{#1\,\delimsize\vert\,#2}
\DeclarePairedDelimiterX\sipm[2]{(}{)_{\mu}}{#1\,\delimsize\vert\,#2}
\DeclarePairedDelimiterX\set[2]{\{}{\}}{#1\,\delimsize\vert\,#2}
\DeclarePairedDelimiterX\dual[2]{\langle}{\rangle}{#1,#2}
\DeclarePairedDelimiterX\sipa[2]{(}{)_{\!_A}}{#1\,\delimsize\vert\,#2}
\DeclarePairedDelimiterX\sipb[2]{(}{)_{\!_B}}{#1\,\delimsize\vert\,#2}
\newcommand{\anti}[1]{\bar{#1}^*}

\newcommand{\limn}{\lim\limits_{n\rightarrow\infty}}
\newcommand{\limsupn}{\limsup\limits_{n\rightarrow\infty}}

\begin{document}
\title{Lebesgue decomposition for representable functionals on $^*$-algebras}

\author[Zs. Tarcsay]{Zsigmond Tarcsay}
\address{Zs. Tarcsay, Department of Applied Analysis, E\"otv\"os L. University, P\'azm\'any P\'eter s\'et\'any 1/c., Budapest H-1117, Hungary; }
\email{tarcsay@cs.elte.hu}


\begin{abstract}
We present a Lebesgue-type decomposition for  a representable functional on a $^*$-algebra into absolutely continuous and singular parts with respect to an other. This generalizes the corresponding results of S. P. Gudder for unital Banach $^*$-algebras. We also prove that the corresponding absolutely continuous parts are absolutely continuous with respect to each other.
\end{abstract}

\maketitle

\section{Introduction}

  S. P. Gudder in \cite{Gudder} presented a Lebesgue-type decomposition theorem for positive functionals on a unital Banach $^*$-algebra $\alg$. In fact, he proved that for given two positive functionals $f,g$ there exist two positive functionals $g_a, g_s$ such that  $g=g_a+g_s$ where $g_a$ is  absolutely continuous with respect to $f$ and $g_s$ is $f$-semi-singular. Here, the concepts of absolute continuity and semi-singularity read as follows: $g$ is called $f$-absolutely continuous ($g\ll f$) if the properties $f(a_n^*a_n)\to0$ and $g((a_n-a_m)^*(a_n-a_m))\to0$ imply $g(a_n^*a_n)\to0$. Furthermore, $g$ is called $f$-semi-singular ($g\perp f$) if there exists a sequence $\seq{a}$ such that $f(a_n^*a_n)\to0$, $g((a_n-a_m)^*(a_n-a_m))\to0$ and $g(a)=\limn g(a_n^*a)$ for any $a\in\alg$. Zs. Sz\H ucs \cite{Szucs2012} proved that the concept of semi-singularity is symmetric in the sense that $g\perp f$ if and only if $f\perp g$. Moreover, $f\perp g$ holds if and only if $h=0$ is the unique positive functional which satisfies $h\leq f$ and $h\leq g$. This latter property is expressed by saying that $f$ and $g$ are mutually singular.

  The main goal of this paper is to establish a Lebesgue decomposition theorem in a more general setting, namely, for representable positive functionals on an arbitrary (not necessarily unital) $^*$-algebra. Being each positive functional of a unital Banach $^*$-algebra representable, this will yield a proper generalization of Gudder's result. A similar generalization was considered by Sz\H ucs \cite{Szucs2013} in developing the Lebesgue decomposition theory of so called representable forms over a complex ($^*$-)algebra.  We will also show that the absolute continuous parts $f_a$ and $g_a$, arising by decomposing $f$ with respect to $g$, and $g$ with respect to $f$, respectively, are absolutely continuous with respect to each other. An analogous statement was proved by T. Titkos \cite{titkos} in the context of nonnegative Hermitian forms.
  
  \section{Preliminaries}

  To begin with we recall briefly the classical Gelfand--Neumark--Segal (GNS) construction which we shall use as a basic tool in our paper. The procedure presented below is  slightly different from what can be find in the literature, see e.g. \cite{lesC*alg, palmer}, or \cite{Sebestyen84}. Why we use this modified version is because  we want to point out the close analogy with the Lebesgue decomposition theory of positive operators in Hilbert spaces, see \cite{tarcsay}. To this aim we introduce first the concept of a positive operator from a vector space into its antidual, cf. \cite{SSZT}. Let $\alg$ be a (not necessarily unital) $^*$-algebra, and denote by $\alg^*$ and $\anti{\alg}$ the algebraic dual and antidual of $\alg$, respectively. Here, the latter one is understood as the vector space of all mappings $\phi:\alg\to \dupC$ satisfying
  \begin{equation*}
    \phi(a+b)=\phi(a)+\phi(b),\qquad \phi(\lambda a)=\overline{\lambda}\phi(a),\qquad a,b\in\alg, \lambda\in\dupC.
  \end{equation*}
  The elements of $\anti{\alg}$ are referred to as antilinear functionals of $\alg$. For $\phi\in\anti{\alg}$ and $a\in\alg$ we shall use the notation
  \begin{equation*}
    \dual{\phi}{a}:=\phi(a).
  \end{equation*}

  In the center of our attention there are those antilinear functionals which derive from a given  positive functional $f$ ($f(a^*a)\geq0, a\in\alg$), defined by the correspondence
  \begin{equation}\label{E:Aa_def}
    \alg\to\dupC, \quad x\mapsto f(x^*a).
  \end{equation}
  The mapping $(a,b)\mapsto f(b^*a)$ defines obviously a semi inner product on $\alg$, hence the Cauchy--Schwarz inequality implies
  \begin{equation}\label{E:CBS}
    \abs{f(b^*a)}^2\leq f(a^*a)f(b^*b), \qquad a,b\in\alg.
  \end{equation}

  We associate now a \emph{positive operator} $A$ with the positive functional $f$: For fixed $a\in\alg$, let $Aa$ denote the functional \eqref{E:Aa_def}. Then $A$ is a linear operator of $\alg$ into $\anti{\alg}$ which is nonnegative definite in the sense that
  \begin{equation*}
    \dual{Aa}{a}=f(a^*a)\geq0,\qquad a\in\alg.
  \end{equation*}
  Observe immediately that $A$ is symmetric:
  \begin{equation*}
    \dual{Aa}{b}=\overline{\dual{Ab}{a}},\qquad a,b\in\alg.
  \end{equation*}

  Hereinafter we make two additional assumptions on $f$: suppose that
  \begin{equation}\label{E:representable}
    \abs{f(a)}^2\leq C\cdot f(a^*a),\qquad a\in\alg,
  \end{equation}
  holds for a nonnegative  constant $C$ and furthermore that
  \begin{equation}\label{E:bounded}
    f(b^*a^*ab)\leq \lambda_a \cdot f(b^*b),\qquad b\in\alg.
  \end{equation}
  holds for any $a\in\alg$ with some $\lambda_a\geq 0$.
  We notice here that \eqref{E:bounded} holds automatically in Banach $^*$-algebras, namely by $\lambda_a=r(a^*a)$, where $r$ stands for the spectral radius.
 As it is well known,  assumptions \eqref{E:representable} and \eqref{E:bounded} express  the representability of the positive functional $f$. That means that there exist a Hilbert space $\hil$, a cyclic vector $\zeta\in\hil$, and a $^*$-representation $\pi$ of $\alg$ in $\bh$ such that
  \begin{equation*}
    f(a)=\sip{\pi(a)\zeta}{\zeta},\qquad a\in\alg.
  \end{equation*}

  Such a triple $(\hil, \pi, \zeta)$ is obtained due to the well known  Gelfand--Neumark--Segal (GNS) construction (see \cite{SSZT}):
  Consider the range space $\ran A$ of the linear operator $A$ in $\anti{\alg}$. This becomes a pre-Hilbert space endowed by the inner product
  \begin{equation}\label{E:sipa}
    \sipa{Aa}{Ab}:=f(b^*a),\qquad a\in\alg.
  \end{equation}
  (Note that the Cauchy--Schwarz inequality \eqref{E:CBS} shows that $\sipa{Aa}{Aa}=0$ implies $Aa=0$ for $a\in\alg$ and hence that $\sipa{\cdot}{\cdot}$ defines an inner product on $\ran A$, indeed.) The completion $\hila$ is then a Hilbert space in which we introduce a densely defined continuous operator $\pia(x)$ for any fixed $x\in\alg$ by letting
  \begin{equation}\label{E:pia}
    \pia(x)(Aa):=A(xa),\qquad a\in\alg.
  \end{equation}
  The continuity of $\pia(x)$ is due to \eqref{E:bounded}:
  \begin{equation*}
    \sipa{A(xa)}{A(xa)}=f(a^*x^*xa)\leq \lambda_x\cdot f(a^*a)=\lambda_x\cdot \sipa{Aa}{Aa}.
  \end{equation*}
  If we continue to write $\pia(x)$ for its unique norm preserving extension then it is easy to verify that $\pia$ is a $^*$-representation of $\alg$ in $\bha$. The cyclic vector of $\pia$ is obtained by considering the linear functional $Aa\mapsto f(a)$ from $\hila$ into $\dupC$ whose continuity is guaranteed by \eqref{E:representable}. The Riesz representation theorem yields then a unique vector $\zeta_A\in\hila$  satisfying
  \begin{equation}\label{E:zeta_A}
    f(a)=\sipa{Aa}{\zeta_A},\qquad a\in\alg.
  \end{equation}
  It is again easy to verify the identity
  \begin{equation}\label{E:piazeta}
    \pia(a)\zeta_A=Aa,\qquad a\in\alg,
  \end{equation}
  whence we infer that
  \begin{equation}\label{E:cyclic}
    f(a)=\sipa{\pia(a)\zeta_A}{\zeta_A},\qquad a\in\alg.
  \end{equation}
  That $\zeta_A$ is a cyclic vector of $\pia$ follows from identity \eqref{E:piazeta}.

\section{Lebesgue decomposition for representable functionals}

Throughout this section we fix two representable positive functionals $f$ and $g$ on the  $^*$-algebra $\alg$. Let $A$ and $B$ stand for the positive operators associated with $f$ and $g$, respectively. The GNS-triplets $(\hila,\pia,\zeta_A)$ and $(\hilb,\pib,\zeta_B)$, induced by $f$ and $g$, respectively, are defined along the construction  of the previous section. Let us recall the notions of absolute continuity and singularity regarding positive functionals (see \cite{Gudder} and \cite{otldoplf}): $g$ is called \emph{absolutely continuous} with respect to $f$ (shortly, $g$ is $f$-absolutely continuous) if
\begin{equation*}
    f(a_n^*a_n)\to0 \quad\textrm{and} \quad g((a_n-a_m)^*(a_n-a_m))\to0\quad \textrm{imply} \quad g(a_n^*a_n)\to0
\end{equation*}
for any sequence $\seq{a}$ of $\alg$. On the other hand, $f$ and $g$ are mutually \emph{singular}   if  the properties $h\leq f$ and $h\leq g$ imply $h=0$ for any representable positive functional $h$.

Our aim in the rest of this section is to establish a Lebesgue decomposition theorem for representable positive functionals. More precisely, we shall show that $g$ splits into a sum $g=g_a+g_s$ where both $g_a$ and $g_s$ are representable positive functionals with $g_a$ $f$-absolutely continuous and $g_s$ $f$-singular.  Such a result was proved by Gudder \cite{Gudder}
for positive functionals on a unital Banach $^*$-algebra and by Sz\H ucs \cite{Szucs2013} in a more general setting, namely for representable forms over a complex algebra.

Our treatment is based on the following observation: If $\seq{a}$ is a sequence from $\alg$ such that
\begin{equation}\label{E:closable_seq}
\sipa{Aa_n}{Aa_n}\to0 \qquad \mbox{and}\qquad \sipb{B(a_n-a_m)}{B(a_n-a_m)}\to 0,
\end{equation}
then $Ba_n\to\zeta$ for some $\zeta\in\hil$. If $g$ is $f$-absolutely continuous then $\zeta$ must be $0$.
We introduce therefore the following  closed linear subspace of $\hilb$ (cf. also \cite{Kosaki, sebestytarcsaytitkos, tarcsay}):
\begin{equation}\label{E:M}
    \M:=\set{\zeta\in\hilb}{\exists \seq{a}\subseteq \alg,\sipa{Aa_n}{Aa_n}\to0, Ba_n\to\zeta~ \mbox{in $\hilb$} }.
\end{equation}
 In fact, $\M$ is nothing but the so called \emph{multivalued part} of the closure of the following linear relation
\begin{equation}\label{E:Bkalap}
    T:=\set{(Aa,Ba)\in\hila\times\hilb}{a\in\alg},
\end{equation}
cf. \cite{Hassi2007} and \cite{tarcsay}. That is to say,
\begin{equation*}
    \M=\mul \overline{T}:=\set{\zeta\in\hilb}{(0,\zeta)\in \overline{T}}.
\end{equation*}
Furthermore, $g$ is $f$-absolutely continuous precisely if $\M=\{0\}$, i.e., when $T$ is (the graph of) a closable operator. We will see below that $\M=\hilb$ holds if and only if $g$ and $f$ are mutually singular. In any other cases, $\M$ is a proper closed $\pi_B$-invariant subspace of $\hilb$ (see Lemma \ref{L:invariant} below).

Let $P$ stand for the orthogonal projection of $\hilb$ onto $\M$, and introduce the functionals $g_a$ and $g_s$ as follows
\begin{equation}\label{E:g_a,g_s}
    g_a(a):=\sipb{\pib(a)(I-P)\zeta_B}{(I-P)\zeta_B},\qquad g_s(a):=\sipb{\pib(a)P\zeta_B}{P\zeta_B},
\end{equation}
for $a\in\alg$.
Our main purpose  is to prove that both $g_a$ and $g_s$ are representable positive functionals such that $g=g_a+g_s$ where $g_a$ is $f$-absolutely continuous and $g_s$ is singular with respect to $f$. Moreover, $g_a$ is maximal in the sense that $h\leq g_a$ holds for each $f$-absolutely continuous representable functional $h$ satisfying $h\leq g$.

\begin{lemma}\label{L:invariant}
    Let $\alg$ be  $^*$-algebra and let $f,g$ be representable functionals of $\alg$. Then $\M$ and $\M^{\perp}$ are both $\pib$-invariant subspaces of $\hilb$, and the following identities hold:
    \begin{enumerate}[\upshape (a)]
      \item $\pib(a)P\zeta_B=P\pib(a)\zeta_B=P(Ba)$, $a\in\alg$,
      \item $\pib(a)(I-P)\zeta_B=(I-P)\pib(a)\zeta_B=(I-P)(Ba)$, $a\in\alg$.
    \end{enumerate}
\end{lemma}
\begin{proof}
    In order to prove the $\pib$-invariancy of $\M$ fix $a\in\alg$ and $\zeta\in\M$, and consider a sequence $\seq{a}$ from $\alg$ satisfying
    \begin{equation*}
        \sipa{Aa_n}{Aa_n}\to0\qquad \mbox{and} \qquad Ba_n\to\zeta~ \mbox{in $\hilb$}.
    \end{equation*}
    Then we have
    \begin{equation*}
        A(aa_n)=\pia(a)(Aa_n)\to0 \qquad \mbox{and} \qquad B(aa_n)=\pib(a)(Ba_n)\to \pib(a)\zeta~ \mbox{in $\hilb$},
    \end{equation*}
    so that $\pi(a)\zeta\in\M$, indeed. Consequently, $\pib(a)\langle\M\rangle\subseteq \M$ for all $a\in\alg$, as claimed. That $\M^{\perp}$ is also $\pib$-invariant follows immediately from the fact that $\pib$ is a $^*$-representation. We are going to prove now (a): for $a\in\alg$ we have $\pib(a)\zeta_B=Ba$ by \eqref{E:piazeta} so it suffices to show the first equality of (a). So fix $\zeta\in\M$; by the $\pib$-invariancy of $\M$ we have that
    \begin{gather*}
        \sipb{P\pib(a)\zeta_B-\pib(a)P\zeta_B}{\zeta}=\sipb{\pib(a)\zeta_B-\pib(a)P\zeta_B}{\zeta}\\
        =\sipb{\pib(a)(I-P)\zeta_B}{\zeta}=0,
    \end{gather*}
    as $\pib(a)(I-P)\zeta_B\in\M^{\perp}$ which yields (a). Assertion (b) is obtained easily from (a).
\end{proof}

As an immediate consequence we have the following
\begin{corollary} Both of the positive functionals $g_a$ and $g_s$ are representable and their sum satisfies
\begin{equation}\label{E:g=ga+gs}
    g=g_a+g_s.
\end{equation}
More precisely, $\piba:=\pib(\cdot)(I-P)$ and $\pibs:=\pib(\cdot)P$ are both $^*$-representations of $\alg$ in the Hilbert spaces $\M^{\perp}$ and $\M$, respectively,  with cyclic vectors $(I-P)\zeta_B$ and $P\zeta_B$, respectively, which satisfy
    \begin{equation}\label{pibapibs}
        g_a(a)=\sipb{\piba(a)(I-P)\zeta_B }{(I-P)\zeta_B},\qquad  g_s(a)=\sipb{\pibs(a)P\zeta_B }{P\zeta_B},
    \end{equation}
for $a\in\alg.$
\end{corollary}

We are now in position to state and prove the main result of the paper, the Lebesgue decomposition theorem of representable functionals:
\begin{theorem}\label{T:maintheorem}
    Let $\alg$ be  $^*$-algebra, $f,g$ representable functionals on $\alg$. Then
    \begin{equation*}
        g=g_a+g_s
    \end{equation*}
    is according to the Lebesgue decomposition, that is to say, both $g_a$ and $g_s$ are representable functionals such that $g_a$ is absolutely continuous with respect to $f$ and that $g_s$ and $f$  are mutually singular. Furthermore, $g_a$ is  maximal in the following sense: $h\leq g$ and $h\ll f$ imply  $h\leq g_a$ for any representable positive functional $h$.
\end{theorem}
\begin{proof}
    We start by proving that $g_a$ is $f$-absolutely continuous. Consider therefore a sequence $\seq{a}$ such that
    \begin{equation*}
        f(a_n^*a_n)\to0,\qquad \mbox{and} \qquad g_a((a_n-a_m)^*(a_n-a_m))\to0.
    \end{equation*}
    Then, by Lemma \ref{L:invariant} we have
    \begin{equation*}
       \sipa{Aa_n}{Aa_n}\to0,\qquad \sipb{(I-P)(B(a_n-a_m))}{(I-P)(B(a_n-a_m))}\to0.
    \end{equation*}
     Nevertheless, the operator $\hila\supseteq\ran A\to\hilb$, $Ax\mapsto(I-P)(Bx)$ coincides with the so called regular part $T_{\textrm{reg}}$ (see \cite[(4.1)]{Hassi2007}) of the linear relation $T$ of \eqref{E:Bkalap} hence it is closable in virtue of \cite[Theorem 4.1]{Hassi2007}. Consequently,
    \begin{equation*}
        g_a(a_n^*a_n)=\sipb{(I-P)(Ba_n)}{(I-P)(Ba_n)}\to0,
    \end{equation*}
    which proves the absolute continuity part of the statement.

    In the next step we prove the extremal property of $g_a$. Consider a representable functional $h$ on $\alg$ such that $h\leq g$ and that $h$ is $f$-absolutely continuous. Then we have by representability
    \begin{align*}
        \abs{h(a)}^2\leq C\cdot h(a^*a)\leq C\cdot g(a^*a)= C\cdot\sipb{Ba}{Ba},\qquad a\in\alg,
    \end{align*}
    hence the linear functional $Ba\mapsto h(a)$ is continuous on $\ran B$ of $\hilb$. The Riesz representation theorem yields therefore a (unique) representing vector $\zeta_h\in\hilb$ that fulfills
    \begin{equation}\label{E:zetah}
        h(a)=\sipb{Ba}{\zeta_h},\qquad a\in\alg.
    \end{equation}
    We state that $\zeta_h\in\M^{\perp}$. Fix therefore $\zeta\in\M$ and consider a sequence $\seq{a}$ from $\alg$ such that
    \begin{align*}
        \sipa{Aa_n}{Aa_n}\to0 \qquad \mbox{\and}\qquad Ba_n\to\zeta~\mbox{in $\hilb$}.
    \end{align*}
    In particular, $\seq{Ba}$ is Cauchy in $\hilb$, therefore $h((a_n-a_m)^*(a_n-a_m))\to0$ holds by $h\leq g$ and thus $h(a_n^*a_n)\to0$ as $h$ is $f$-absolutely continuous. That implies that
    \begin{align*}
        \abs{\sipb{\zeta}{\zeta_h}}^2&=\limn\abs{\sipb{Ba_n}{\zeta_h}}^2=\limn \abs{h(a_n)}^2\leq C\cdot\limn h(a_n^*a_n)\to0,
    \end{align*}
    which yields the desired identity. Fix now $a\in\alg$; by Lemma \ref{L:invariant} and according to identity $(I-P)\zeta_h=\zeta_h$ we conclude that
    \begin{align*}
        h(a^*a)&=\sipb{B(a^*a)}{\zeta_h}=\sipb{(I-P)(Ba)}{\pib(a)\zeta_h}\\
               &\leq \|(I-P)(Ba)\|_{\!_B} \|\pib(a)\zeta_h\|_{\!_B}=\sqrt{g_a(a^*a)}\|\pib(a)\zeta_h\|_{\!_B},
    \end{align*}
    thus $h\leq g_a$ will be obtained once we prove that
    \begin{equation}\label{E:pibazetah}
        \sipb{\pib(a)\zeta_h}{\pib(a)\zeta_h}\leq h(a^*a),\qquad a\in\alg.
    \end{equation}
    By using the density of $\ran B$ in $\hilb$ it follows that
    \begin{align*}
        \sipb{\pib(a)\zeta_h}{\pib(a)\zeta_h}&=\sup\set[\big]{\abs{\sipb{Bx}{\pib(a)\zeta_h}}^2}{x\in\alg, \sipb{Bx}{Bx}\leq1}\\
                                           &=\sup\set[\big]{\abs{\sipb{B(a^*x)}{\zeta_h}}^2}{x\in\alg, g(x^*x)\leq1}\\
                                           &=\sup\set{\abs{h(a^*x)}^2}{x\in\alg, g(x^*x)\leq1}\\
                                           &\leq \sup\set{h(a^*a)h(x^*x)}{x\in\alg, g(x^*x)\leq1}\\
                                           &\leq h(a^*a),
    \end{align*}
    as it is claimed.

    There is nothing left but to prove that $g_s$ and $f$ are singular with respect to each other. Fix therefore a representable functional $h$ of $\alg$ such that $h\leq f$ and $h\leq g_s$. Then clearly $h\leq g$ and $h$ is $f$-absolutely continuous. By the previous step, there exists $\zeta_h\in\M^{\perp}$ such that \eqref{E:zetah} holds. By density of $\ran B$ we may choose $\seq{a}$ from $\alg$ such that $Ba_n\to\zeta_h$ in $\hilb$. Then we find that
    \begin{align*}
        \abs{\sipb{\zeta_h}{\zeta_h}}^2&=\limn\abs{\sipb{Ba_n}{\zeta_h}}^2=\limn\abs{h(a_n)}^2\leq C\cdot\limsupn h(a_n^*a_n)\\
                                 &\leq  C\cdot\limsupn g_s(a_n^*a_n)=C\cdot\limsupn \sipb{P(Ba_n)}{P(Ba_n)}\\
                                 &= C\cdot \sipb{P\zeta_h}{P\zeta_h}=0,
    \end{align*}
    hence $h=0$. The proof is therefore complete.
\end{proof}

\section{Mutually absolute continuity of the absolute continuous parts}

Let $f,g$ be representable positive functionals on the $^*$-algebra $\alg$ and consider the Lebesgue decompositions
\begin{equation*}
    f=f_a+f_s,\qquad g=g_a+g_s,
\end{equation*}
where $f_a,f_s$ and $g_a,g_s$ are obtained along the procedure presented in the previous section. In accordance with Theorem \ref{T:maintheorem},  $f_a\ll g$ and $g_a\ll f$. Our purpose in this section is to show that the absolute continuous parts $f_a$ and $g_a$ are mutually absolute continuous, that is that $f_a\ll g_a$ and $g_a\ll f_a$ hold true. The heart of the matter is in the following lemma which may be of interest on its own right.
\begin{lemma}\label{L:invertible}
    Let $T$ be a linear relation between two Hilbert spaces $\hil$ and $\kil$. Let $\overline{T}$ stand for the closure of $T$ and let $P,Q$ be the orthogonal projections onto $\ker \overline{T}$ and $\mul \overline{T}$, respectively. Then
    \begin{equation*}
        S_0:=\set{((I-P)h,(I-Q)k)}{(h,k)\in T}
    \end{equation*}
    is (the graph of) a closable linear operator whose closure $\overline{S_0}$ is one-to-one.
\end{lemma}
\begin{proof}
    We shall show that
    \begin{equation*}
        S:=\set{((I-P)h,(I-Q)k)}{(h,k)\in \overline{T}}
    \end{equation*}
    is an invertible closed operator. As $S_0\subseteq S$, this contains our original assertion. Consider first the so called regular part $(\overline{T})_{\textrm{reg}}$ of $\overline{T}$, which is defined by
    \begin{equation*}
        (\overline{T})_{\textrm{reg}}:=\set{(h,(I-Q)k)}{(h,k)\in \overline{T}},
    \end{equation*}
    cf. \cite{Hassi2007}. Let us denote it by $R$ for the sake of brevity. We claim first that $R$ is a closed linear operator such that $R\subseteq\overline{T}$. The proof of this statement can be found in \cite{Hassi2007}, we include here a short proof however, for the sake of the reader.  It is seen easily that $\{0\}\times\mul\overline{T}\subseteq \overline{T}$, and that $\overline{T}-(\{0\}\times\mul\overline{T})=R$. Consequently, $R\subseteq\overline{T}$ and, as  $\{0\}\times\mul\overline{T}$ and $R$ are orthogonal to each other, we infer that $R=\overline{T}\ominus(\{0\}\times\mul\overline{T})$, hence $R$  is closed. To see that $R$ is an operator assume that $(0,(I-Q)k)\in R$ where $(0,k)\in \overline{T}$. This implies that $k\in\mul \overline{T}$ whence $(I-Q)k=0$, that is, $R$ is an operator. Observe furthermore that $\ker R=\ker \overline{T}$: indeed,  by the definition of $R$ it is clear that $\ker \overline{T}\subseteq\ker R $, and the converse inclusion is due to $R\subseteq \overline{T}$.

    Consider now the relation $(R^{-1})_{\textrm{reg}}$. Then $(R^{-1})_{\textrm{reg}}$ is a a closed linear operator due to the above reasoning, such that $(R^{-1})_{\textrm{reg}}\subseteq R^{-1}$. At the same time,
    \begin{equation*}
        \mul R^{-1}=\ker R=\ker \overline{T},
    \end{equation*}
    whence
    \begin{align*}
        (R^{-1})_{\textrm{reg}}&=\set{(k',(I-P)h')}{(h',k')\in R}\\
        &=\set{((I-Q)k,(I-P)h)}{(h,k)\in \overline{T}}=S^{-1}.
    \end{align*}
    We conclude  therefore that $S^{-1}$ is a closed operator, such that $S^{-1}\subseteq R^{-1}$, or equivalently,  $S\subseteq R$. Hence $S$ is an operator as well.
\end{proof}

\begin{theorem}
    Let $f$ and $g$ be representable positive functionals on the $^*$-algebra $\alg$. Denote by $f_a$ and $g_a$ the $g$-absolutely continuous and the  $f$-absolutely continuous parts of $f$ and $g$, respectively. Then $f_a$ and  $g_a$ are absolutely continuous with respect to each other: $f_a\ll g_a$ and $g_a\ll f_a$.
\end{theorem}
\begin{proof}
    Consider the linear relation $T$ of \eqref{E:Bkalap} and let $P,Q$ be the orthogonal projections onto $\ker \overline{T}$ and $\mul \overline{T}$, respectively. By Theorem \ref{T:maintheorem}, the corresponding absolute continuous parts satisfy
    \begin{equation*}
        f_a(a^*a)=\|(I-P)Aa\|_A^2,\qquad g_a(a^*a)=\|(I-Q)Ba\|_B^2,\qquad a\in\alg.
    \end{equation*}
     According to Lemma \ref{L:invertible}, the relation
    \begin{equation*}
        S:=\set{((I-P)Aa,(I-Q)Ba)}{a\in\alg}
    \end{equation*}
    is (the graph of) a closable operator. Hence, if
    \begin{equation*}
        f_a(a_n^*a_n)=\|(I-P)Aa_n\|_A^2\to0,
    \end{equation*}
    and
    \begin{equation*}
         g_a((a_n-a_m)^*(a_n-a_m))=\|(I-Q)B(a_n-a_m)\|_B^2\to0
    \end{equation*}
    hold for some $\seq{a}$ then
    \begin{equation*}
        g_a(a_n^*a_n)=\|(I-Q)Ba_n\|_B^2\to0,
    \end{equation*}
    whence we deduce that $g_a$ is $f_a$-absolutely continuous. That $f_a$ is $g_a$-absolutely continuous follows from the fact that $S$ is one-to-one with closable inverse, according again to Lemma \ref{L:invertible}.
\end{proof}
\bibliographystyle{amsplain}

\end{document}